\documentclass[graybox]{svmult}

\usepackage{mathptmx}       
\usepackage{helvet}         
\usepackage{courier}        
\usepackage{makeidx}         
\usepackage{graphicx}        
\usepackage{multicol}        
\usepackage[bottom]{footmisc}
\usepackage{amsmath,amssymb,amsfonts}        
\makeindex             

\usepackage{tikz}

\usetikzlibrary{matrix,arrows,calc}

\newcommand{\maps}{\colon}
\newcommand{\tensor}{\otimes}
\newcommand{\htensor}{\widehat{\otimes}}

\newcommand{\pr}{\mathrm{pr}}

\newcommand{\bs}{\mathbf{s}}

\newcommand{\xto}[1]{\xrightarrow{#1}}

\newcommand{\Z}{\mathbb{Z}}

\newcommand{\cF}{\mathcal{F}}
\newcommand{\cC}{\mathcal{C}}

\newcommand{\fF}{\mathfrak{F}}

\newcommand{\fib}{\twoheadrightarrow}
\newcommand{\trivfib}{\overset{\sim}{\twoheadrightarrow}}
\newcommand{\weq}{\xrightarrow{\sim}}

\newcommand{\del}{\partial}
\newcommand{\bul}{\bullet}
\newcommand{\und}[1]{{\underline{#1}}}
\newcommand{\ov}[1]{{\overline{#1}}}

\newcommand{\ti}[1]{{\tilde{#1}}}

\newcommand{\curv}{{\mathsf{curv}}}
\newcommand{\horn}[2]{\Lambda^{#1}_{#2}}

\DeclareMathOperator{\MC}{\mathrm{MC}}
\DeclareMathOperator{\sMC}{\mathfrak{MC}}

\DeclareMathOperator{\FLie}{\widehat{\mathsf{Lie}}_\infty}
\DeclareMathOperator{\Kan}{\mathsf{Kan}}

\DeclareMathOperator{\Hom}{\mathrm{Hom}}

\DeclareMathOperator{\End}{\mathrm{End}}

\DeclareMathOperator{\tow}{\mathsf{tow}}
\DeclareMathOperator{\sSet}{\mathsf{sSet}}
\DeclareMathOperator{\Cobar}{\mathrm{Cobar}}

\DeclareMathOperator{\Cyl}{\mathrm{Cyl}}
\DeclareMathOperator{\Conv}{\mathrm{Conv}}

\begin{document}
\title*{Homotopical properties of the simplicial Maurer--Cartan functor}

\author{Christopher L.\  Rogers}
\institute{
 Department of Mathematics and Statistics, University of Nevada,
 Reno. 1664 N. Virginia Street Reno, NV 89557-0084 USA, 
\email{chrisrogers@unr.edu, chris.rogers.math@gmail.com}}

\maketitle

\abstract{
  We consider the category whose objects are filtered, or complete,
  $L_\infty$-algebras and whose morphisms are $\infty$-morphisms which
  respect the filtrations. We then discuss the homotopical properties
  of the Getzler--Hinich simplicial Maurer--Cartan functor which
  associates to each filtered $L_\infty$-algebra a Kan simplicial set,
  or $\infty$-groupoid. In previous work with V.\ Dolgushev, we showed
  that this functor sends weak equivalences of filtered
  $L_\infty$-algebras to weak homotopy equivalences of simplicial
  sets. Here we sketch a proof of the fact that this functor also
  sends fibrations to Kan fibrations. To the best of our knowledge, 
  only special cases of this result have previously appeared in the literature.
  As an application, we show how these
  facts concerning the simplicial Maurer--Cartan functor
  provide a simple $\infty$-categorical formulation of
  the Homotopy Transfer Theorem.
}

\section{Introduction}
Over the last few years, there has been increasing interest in the
homotopy theory of filtered, or complete,
$L_\infty$-algebras\footnote{ Throughout this paper, all algebraic
  structures have underlying $\Z$ graded $\Bbbk$-vector spaces with
  char $\Bbbk =0$.} and the role these objects play in deformation
theory \cite{Getzler,Hinich}, rational homotopy theory
\cite{Berglund,Buijs,Lazarev}, and the homotopy theory of homotopy
algebras \cite{HAforms,Enhanced,Dotsenko}. One important tool used in these applications is the
simplicial Maurer--Cartan functor $\sMC_{\bul}(-)$ which produces from any
filtered $L_\infty$-algebra a Kan simplicial set, or
$\infty$-groupoid. This construction, first appearing in the work of
V.\ Hinich \cite{Hinich} and E.\ Getzler \cite{Getzler}, can (roughly) be thought of as a ``non-abelian
analog'' of the Dold--Kan functor from chain complexes to simplicial
vector spaces.  In deformation theory, these $\infty$-groupoids give
higher analogs of the Deligne groupoid. In rational homotopy theory,
this functor generalizes the Sullivan realization functor, and has
been used to study rational models of mapping spaces.

A convenient presentation of the homotopy theory of filtered
$L_\infty$-algebras has yet to appear in the literature. But based on
applications, there are good candidates for what the weak equivalences
and fibrations should be between such objects. One would also hope that the
simplicial Maurer--Cartan functor sends these morphisms to weak
homotopy equivalences and Kan fibrations, respectively. For various special cases, which are
recalled in Section \ref{sec:sMC}, it is known that this is indeed
true. In joint work with V.\ Dolgushev \cite{GM_Theorem}, we
showed that, in general, $\sMC_{\bul}(-)$ maps any weak equivalence of
filtered $L_\infty$-algebras to a weak equivalence of Kan
complexes. This can be thought of as the natural $L_\infty$
generalization of the Goldman--Millson theorem in deformation
theory.

The purpose of this note is to sketch a proof of the analogous result for
fibrations (Thm.\ \ref{thm:fib} in Sec.\ \ref{sec:sMC} below): 
The simplicial Maurer--Cartan functor maps any fibration
between any filtered $L_\infty$-algebras to a fibration between their
corresponding Kan complexes. Our proof is not a simple generalization
of the special cases already found in the literature, nor does it
follow directly from general abstract homotopy theory.
It requires some technical calculations involving Maurer--Cartan
elements, similar to those found in our previous work
\cite{GM_Theorem}.  

As an application, we show in Sec.\ \ref{sec:HTT} that
``$\infty$-categorical'' analogs of the existence and uniqueness
statements that comprise the Homotopy Transfer Theorem
\cite{BM,Berglund2,LV,Markl} follow as a corollary of our Theorem \ref{thm:fib}. 
In more detail, suppose we are given a cochain complex $A$, a
homotopy algebra $B$ of some particular type (e.g, an $A_\infty$,
$L_\infty$, or $C_\infty$-algebra) and a quasi-isomorphism of
complexes $\phi \maps A \to B$.  Then, using the simplicial
Maurer--Cartan functor, we can naturally produce an $\infty$-groupoid
$\fF$ whose objects correspond to solutions to the ``homotopy transfer
problem''. By a solution, we mean a pair consisting of a homotopy
algebra structure on $A$, and a lift of $\phi$ to a
$\infty$-quasi-isomorphism of homotopy algebras $A \weq B$. The fact
that $\sMC_{\bul}(-)$ preserves both weak equivalences and fibrations allows
us to conclude that: (1) The $\infty$-groupoid $\fF$ is non--empty,
and (2) it is contractible. In other words, a homotopy equivalent
transferred structure always exists, and this structure is unique in the strongest possible sense.

\section{Preliminaries} \label{sec:prelim}

\subsection{Filtered $L_\infty$-algebras} \label{sec:FLie}
In order to match conventions in our previous work \cite{GM_Theorem}, we define an
{\bf $L_\infty$-algebra} to be a cochain complex $(L,\del)$ for
which the reduced cocommutative coalgebra $\und{S}(L)$ is equipped with a
degree 1 coderivation $Q$ such that $Q(x)=\del x$ for all $x \in L$
and $Q^2=0$. This structure is equivalent to specifying a sequence of
\underline{degree one} multi-brackets
\begin{equation}
\label{m-bracket}
\{~,~, \dots, ~\}_m : S^m(L) \to L \quad m \geq 2
\end{equation}
satisfying compatibility conditions with the differential $\del$ and
higher--order Jacobi-like identities. (See Eq.\ 2.5. in \cite{GM_Theorem}.)
More precisely, if $\pr_L \maps \und{S}(L) \to L$ denotes the usual
projection, then
\[
\{x_1, x_2, \dots, x_m\}_m = \pr_L  Q ( x_1 x_2 \dots x_m)\,, \qquad \forall x_j \in L\,.
\]
This definition of $L_\infty$-algebra is a ``shifted version'' of the
original definition of $L_\infty$-algebra. A shifted $L_\infty$-structure
on $L$ is equivalent to a traditional $L_\infty$-structure on $\bs L$,
the suspension of $L$.

A \textbf{morphism} (or $\infty$-morphism) $\Phi$ from an
$L_\infty$-algebra $(L,Q)$ to an $L_\infty$-algebra $(\ti{L},\ti{Q})$
is a dg coalgebra morphism
\begin{equation} \label{eq:morphism} \Phi \maps \bigl( \und{S}(L),Q
  \bigr) \to \bigl( \und{S}(\ti{L}), \ti{Q} \bigr).
\end{equation}
Such a morphism $\Phi$ is uniquely determined by its composition with
the projection to $\ti{L}$:
\[
\Phi' := \pr_{\ti{L}} \Phi.
\]
Every such dg coalgebra morphism induces a map of cochain complexes, e.g., the linear term of $\Phi$:
\begin{equation} \label{eq:tangent} \phi:= \pr_{\ti{L}} \Phi \vert_{L}
  \maps (L,\del) \to (\ti{L},\ti{\del}),
\end{equation}
and we say $\Phi$ is \textbf{strict} iff it consists only of a linear term, i.e.
\begin{equation} \label{eq:strict}
\Phi'(x)= \phi(x) \qquad \Phi'(x_1,\ldots,x_m)=0 \quad \forall m \geq 2
\end{equation}

A morphism $\Phi \maps (L,Q) \to (\ti{L},\ti{Q})$ of
$L_\infty$-algebras is an \textbf{$\infty$-quasi-isomorphism} iff
$\phi \maps (L,\del) \to (\ti{L},\ti{\del})$ is a quasi--isomorphism
of cochain complexes.

We say an $L_\infty$-algebra $(L,Q)$ is a \textbf{filtered $L_\infty$-algebra} iff 
the underlying cochain complex $(L, \del)$ is equipped with a complete 
descending filtration,
\begin{equation}
\label{filtr-L}
L = \cF_{1}L \supset \cF_{2}L \supset  \cF_{3}L  \cdots
\end{equation}
\begin{equation}
\label{L-complete}
L =\varprojlim_{k} L/\cF_{k}L\,,
\end{equation}
which is compatible with the brackets, i.e.
\[
\Bigl \{\cF_{i_{1}}L,\cF_{i_{2}}L,\ldots,\cF_{i_{m}}L \Bigr \}_m \subseteq
\cF_{i_{1} + i_{2} + \cdots + i_{m}} L \quad \forall ~~ m >1.
\]
A filtered $L_\infty$-algebra in our sense is a shifted analog of a 
``complete'' $L_\infty$-algebra, in the sense of A.\ Berglund \cite[Def.\ 5.1]{Berglund}.

\begin{remark} \label{rmk:nilpot}
Due to its compatibility with the filtration,
the $L_\infty$-structure on $L$ induces a filtered $L_\infty$-structure on the
quotient $L/\cF_n L$. In particular, $L/\cF_nL$
is a \textbf{nilpotent $L_\infty$-algebra} \cite[Def.\ 2.1]{Berglund}
\cite[Def.\ 4.2]{Getzler}. Moreover, when the induced $L_\infty$-structure is
restricted to the sub-cochain complex 
\[
\cF_{n-1} L/\cF_{n} L \subseteq L/\cF_{n} L 
\]
all brackets of arity $\geq 2$ vanish. Hence, the nilpotent $L_\infty$-algebra  $\cF_{n-1} L/\cF_{n} L$
is an \textbf{abelian $L_\infty$-algebra}.
\end{remark}

\begin{definition}
  We denote by $\FLie$ the category whose objects are filtered
  $L_\infty$-algebras and whose morphisms are $\infty$-morphisms $\Phi
  \maps (L,Q) \to (\ti{L},\ti{Q})$ which are compatible with the
  filtrations:
\begin{equation}
\label{eq:fmorph}
\Phi'( \cF_{i_1} L \otimes  \cF_{i_2} L \otimes \dots \otimes  \cF_{i_m} L) \subset \cF_{i_1 + i_2 + \dots + i_m} \ti{L}\,,
\end{equation}

\end{definition}

\begin{definition}\label{def:weq}
Let $\Phi \maps (L,Q) \to (\ti{L},\ti{Q})$ be a morphism in $\FLie$.
\begin{enumerate}

\item We say $\Phi$ is a \textbf{weak equivalence} iff its linear term $\phi \maps (L, \del) \to (\ti{L},\ti{\del})$ induces a quasi--isomorphism of cochain complexes
\[
\phi \vert_{\cF_n L} \maps (\cF_{n}L, \del) \to (\cF_n \ti{L}, \ti{\del}) \qquad \forall n \geq 1.
\]

\item We say $\Phi$ is a \textbf{fibration} iff its linear term $\phi \maps (L, \del) \to (\ti{L},\ti{\del})$ induces a surjective map of cochain complexes
\[
\phi \vert_{\cF_n L} \maps (\cF_{n}L, \del) \to (\cF_n \ti{L}, \ti{\del}) \qquad \forall n \geq 1.
\]

\item We say $\Phi$ is an \textbf{acyclic fibration} iff $\Phi$ is both a weak equivalence and a fibration.

\end{enumerate}
\end{definition}

\begin{remark} \label{rmk:ses}
If $(L,Q)$ is a filtered $L_\infty$-algebra, then for
each $n \geq 1$, we have the obvious short exact sequence of cochain
complexes
\begin{equation} \label{eq:ses}
0 \to \cF_{n-1}L/\cF_n L \xto{i_{n-1}} L / \cF_{n}L \xto{p_n}  L / \cF_{n-1}L \to 0.
\end{equation}
It is easy to see that \eqref{eq:ses} lifts to a sequence of filtered
$L_\infty$-algebras, in which all of the algebras in the sequence are nilpotent $L_\infty$-algebras
(see Remark \ref{rmk:nilpot}), and in which all of the morphisms in the
sequence are strict. In particular, the morphism $L / \cF_{n}L
\xto{p_n}  L / \cF_{n-1}L$ is a fibration.

\end{remark}

\subsection{Maurer--Cartan elements} \label{sec:MC}
Our reference for this section is Section 2 of \cite{Enhanced}. We refer the reader there for 
details. Let $L$ be a filtered $L_\infty$-algebra. Since $L=\cF_1L$, the compatibility of the
multi--brackets with the filtrations gives us
well defined map of sets $\curv \maps
L^0 \to L^1$: 
\begin{equation} \label{eq:curv}
\curv(\alpha)=\del \alpha + \sum_{m \geq 1} \frac{1}{m!}
\{\alpha^{\tensor m}\}_m.
\end{equation}
Elements of the set
\[
\MC(L):= \{ \alpha \in L^0 ~ \vert ~ \curv(\alpha)=0 \} 
\]
are called the \textbf{Maurer--Cartan (MC) elements} of $L$. Note that
MC elements of $L$ are elements of degree 0. Furthermore, if
$\Phi \maps (L,Q) \to (\ti{L},\ti{Q})$ is a morphism in $\FLie$ then
the compatibility of $\Phi$ with the filtrations allows us to define a
map of sets
\begin{equation}\label{eq:MC_phi}
\begin{split}
\Phi_\ast & \maps \MC(L)  \to \MC(\ti{L})\\
\Phi_\ast(\alpha) &:= \sum_{m \geq 2} \frac{1}{m!} \Phi'(\alpha^{\tensor
  m}).
\end{split}
\end{equation}
The fact that $\curv(\Phi(\alpha))=0$ is proved in \cite[Prop.\ 2.2]{Enhanced}.

Given an MC element $\alpha \in \MC(L)$, we can ``twist'' the
$L_\infty$-structure on $L$, to obtain a new filtered
$L_\infty$-algebra $L^\alpha$.  As a graded vector space with a
filtration, $L^{\alpha} = L$; the differential $\del^{\alpha}$ and the
multi-brackets $\{~,\dots, ~ \}^{\alpha}_m$ on $L^{\alpha}$ are defined by
the formulas
\begin{equation}
\label{eq:diff-twisted}
\del^{\alpha}(v) : = \del (v)  + \sum_{k=1}^{\infty} \frac{1}{k!} \{\alpha, \dots, \alpha, v\}_{k+1}\,,
\end{equation}
\begin{equation}
\label{eq:Linft_twisted}
\{v_1,v_2,\cdots, v_m \}^{\alpha}_m : =
\sum_{k=0}^{\infty} \frac{1}{k!} \{\alpha, \dots, \alpha, 
v_1, v_2,\cdots, v_m \}_{k+m}\,.
\end{equation}

\subsection{Getzler--Hinich construction} \label{sec:GH}
The MC elements of $(L,Q)$ are in fact the vertices of a simplicial
set. Let $\Omega_n$ denote the de Rham-Sullivan algebra of polynomial 
differential forms on the geometric simplex $\Delta^n$ with 
coefficients in $\Bbbk$.  The simplicial set $\sMC_\bul(L)$ is
 defined as
\begin{equation}
\label{int-mMC}
\sMC_n(L) : = \MC(L \htensor \Omega_n)
\end{equation}
where $L \htensor \Omega_n$ is the filtered $L_\infty$-algebra
defined as the projective limit of nilpotent $L_\infty$-algebras
\[
L \htensor \Omega_n : = \varprojlim_{k} \big( (L/\cF_{k}L) \otimes \Omega_n \big).
\]
Recall that the $L_\infty$-structure on the tensor product of chain complexes 
$(L/\cF_{k}L) \otimes \Omega_n$ is induced by the structure on
$L/\cF_{k}L$, and is well--defined since $\Omega_n$ is a commutative
algebra. For example:
\[
\{\bar{x}_1 \tensor \omega_1, \bar{x}_2 \tensor \omega_2, 
\ldots,\bar{x}_l \tensor \omega_l\}:= \pm  
\{\bar{x}_1, \bar{x}_2, \ldots,\bar{x}_l \}\tensor
\omega_1\omega_2\cdots \omega_l.
\]
Proposition 4.1 of \cite{Enhanced} implies that the simplicial set
$\sMC_{\bul}(L)$ is a Kan complex, or $\infty$-groupoid, which is
sometimes referred to as the ``Deligne--Getzler--Hinich''
$\infty$-groupoid of $L$.

Any morphism $\Phi \maps L \to \ti{L}$ in $\FLie$ induces a morphism 
$\Phi^{(n)} \maps L \htensor \Omega_n \to \ti{L}\htensor \Omega_n$ for
each $n \geq 0$ in the obvious way:
\begin{equation} \label{eq:phi-n}
\Phi^{(n)}(x_1 \tensor \theta_1,x_2 \tensor \theta_2, \hdots, x_m
\tensor \theta_m):= \pm \Phi(x_1,x_2,\ldots,x_m)\tensor \theta_1
\theta_2 \cdots \theta_m.
\end{equation}
This then gives us a map of MC sets $\Phi^{(n)}_\ast \maps 
\MC(L \htensor \Omega_n) \to \MC (\ti{L} \htensor \Omega_n)$ defined
via Eq.\ \ref{eq:MC_phi}. 
It is easy to see that $\Phi^{(n)}$ is compatible with the face and
degeneracy maps, which leads us to the \textbf{simplicial
  Maurer--Cartan functor}
\begin{equation} \label{eq:MCfunctor}
\begin{split}
\sMC_{\bul} &\maps \FLie \to \Kan\\
\sMC_{\bul} \Bigl ( L \xto{\Phi} \ti{L} \Bigr) &:= \sMC_\bul(L)
\xto{\Phi^{\bul}_\ast} \sMC_\bul(\ti{L})
\end{split}
\end{equation}

\section{The functor $\sMC_\bul(-)$ preserves weak equivalences and
  fibrations} \label{sec:sMC}

Our first observation concerning the simplicial Maurer--Cartan functor
is that it sends a weak equivalence $\Phi \maps L \weq \ti{L}$ in
$\FLie$ to a weak homotopy equivalence.  
For the special case in which
$\Phi$ is a \textit{strict} quasi-isomorphism between (shifted) \textit{dg Lie algebras}, V.\ Hinich \cite{Hinich} showed that $\sMC_{\bul}(\Phi)$ is a weak
equivalence. If $\Phi$ happens to be a \textit{strict}
quasi-isomorphism between \textit{nilpotent} $L_\infty$-algebras, then 
E.\ Getzler \cite{Getzler} showed that $\sMC_{\bul}(\Phi)$ is a weak equivalence. 
The result for the general case of $\infty$-quasi-isomorphisms between filtered
$L_\infty$-algebras was proved in our previous work with V.\ Dolgushev.

\begin{theorem}\cite[Thm.\ 1.1]{GM_Theorem} \label{thm:GM} 
If $\Phi \maps (L,Q) \weq
  (\ti{L},\ti{Q})$ is a weak equivalence of filtered
  $L_\infty$-algebras, then
\[
\sMC_{\bul}(\Phi) \maps \sMC_{\bul}(L) \to \sMC_{\bul}(\ti{L})
\]
is a homotopy equivalence of simplicial sets.
\end{theorem} 
It is interesting that the most subtle part of the proof of the above theorem involves establishing the bijection between $\pi_0 (\sMC_{\bul}(L))$ and $\pi_0( \sMC_{\bul}(\ti{L}))$.

The second noteworthy observation is that if $\Phi \maps L \fib \ti{L}$ is a fibration
then $\sMC_{\bul}(\Phi)$ is a Kan fibration.
To the best of our knowledge, this result, at this level of generality, is new.

\begin{theorem} \label{thm:fib} 
If $\Phi \maps (L,Q) \to
  (\ti{L},\ti{Q})$ is a fibration of filtered
  $L_\infty$-algebras, then
\[
\sMC_{\bul}(\Phi) \maps \sMC_{\bul}(L) \to \sMC_{\bul}(\ti{L})
\]
is a fibration of simplicial sets.
\end{theorem} 
Two special cases of Thm.\ \ref{thm:fib} already exist in the literature.  
If $\Phi$ happens to be a \textit{strict}
fibration between \textit{nilpotent} $L_\infty$-algebras, then the result is
again due to E.\ Getzler \cite[Prop.\ 4.7]{Getzler}. If $\Phi$ is a \textit{strict}
fibration between \textit{profinite} filtered $L_\infty$-algebra, then
S.\ Yalin showed \cite[Thm.\ 4.2(1)]{Yalin} that $\sMC_{\bul}(\Phi)$ is a fibration.

The proof of Thm.\ \ref{thm:fib} is technical and will appear in
elsewhere in full detail. We give a sketch here. 

Suppose
$\Phi \maps L \to \ti{L}$ is a fibration. This induces a morphism
between towers of nilpotent $L_\infty$-algebras
\begin{equation}\label{eq:L_tow}
\begin{tikzpicture}[descr/.style={fill=white,inner sep=2.5pt},baseline=(current  bounding  box.center)]
\matrix (m) [matrix of math nodes, row sep=2em,column sep=3em,
  ampersand replacement=\&]
  {  
\cdots \&  L / \cF_{n+1}L \& L / \cF_{n}L  \&  L / \cF_{n-1}L \& \cdots \& 0 \\
\cdots \&  \ti{L} / \cF_{n+1}\ti{L} \& \ti{L} / \cF_{n} \ti{L}  \&  \ti{L} / \cF_{n-1} \ti{L} \& \cdots \& 0 \\
}; 

\path[->,font=\scriptsize] 
  (m-1-2) edge node[auto,swap] {$ \ov{\Phi} $} (m-2-2)
  (m-1-3) edge node[auto] {$ \ov{\Phi} $} (m-2-3)
  (m-1-4) edge node[auto] {$ \ov{\Phi} $} (m-2-4)
;

\path[->>,font=\scriptsize] 
(m-1-1) edge node[auto] {$$} (m-1-2)
(m-2-1) edge node[auto] {$$} (m-2-2)
(m-1-4) edge node[auto] {$$} (m-1-5)
(m-2-4) edge node[auto] {$$} (m-2-5)
(m-1-5) edge node[auto] {$$} (m-1-6)
(m-2-5) edge node[auto] {$$} (m-2-6)
(m-1-6) edge node[auto] {$$} (m-2-6)

(m-1-2) edge node[auto] {$p_{n+1}$} (m-1-3)
(m-2-2) edge node[auto] {$\ti{p}_{n+1}$} (m-2-3)
(m-1-3) edge node[auto] {$p_n$} (m-1-4)
(m-2-3) edge node[auto] {$\ti{p}_n$} (m-2-4)
;
\end{tikzpicture}
\end{equation}
which gives us a morphism between towers of Kan complexes:
\begin{equation}\label{eq:S_tow}
\begin{tikzpicture}[descr/.style={fill=white,inner sep=2.5pt},baseline=(current  bounding  box.center)]
\matrix (m) [matrix of math nodes, row sep=2em,column sep=3em,
  ampersand replacement=\&]
  {  
\cdots \&  \sMC_{\bul} \Bigl ( L / \cF_{n+1}L \Bigr) \& \sMC_{\bul} \Bigl ( L / \cF_{n}L \Bigr)   \& \sMC_{\bul} \Bigl (  L / \cF_{n-1}L \Bigr)  \& \cdots \& \ast \\
\cdots \&  \sMC_{\bul} \Bigl ( \ti{L} / \cF_{n+1}\ti{L}  \Bigr) \& \sMC_{\bul} \Bigl (\ti{L} / \cF_{n} \ti{L}  \Bigr)   \&  \sMC_{\bul} \Bigl (\ti{L} / \cF_{n-1} \ti{L}  \Bigr)  \& \cdots \& \ast \\
}; 

\path[->,font=\scriptsize] 
  (m-1-2) edge node[auto,swap] {$ \ov{\Phi}^{(\bul)}_{\ast} $} (m-2-2)
  (m-1-3) edge node[auto] {$ \ov{\Phi}^{(\bul)}_{\ast} $} (m-2-3)
  (m-1-4) edge node[auto] {$ \ov{\Phi}^{(\bul)}_{\ast} $} (m-2-4)
  
;

\path[->>,font=\scriptsize] 
(m-1-1) edge node[auto] {$$} (m-1-2)
(m-2-1) edge node[auto] {$$} (m-2-2)
(m-1-4) edge node[auto] {$$} (m-1-5)
(m-2-4) edge node[auto] {$$} (m-2-5)
(m-1-5) edge node[auto] {$$} (m-1-6)
(m-2-5) edge node[auto] {$$} (m-2-6)
(m-1-6) edge node[auto] {$ $} (m-2-6)

(m-1-2) edge node[auto] {$p_{n+1 \ast}^{(\bul)}$} (m-1-3)
(m-2-2) edge node[auto] {$\ti{p}_{n+1 \ast}^{(\bul)}$} (m-2-3)
(m-1-3) edge node[auto] {$p_{n \ast}^{(\bul)}$} (m-1-4)
(m-2-3) edge node[auto] {$\ti{p}^{(\bul)}_{n \ast}$} (m-2-4)
;
\end{tikzpicture}
\end{equation}

The morphisms $p_n$ and $\ti{p}_n$ are strict fibrations between nilpotent $L_\infty$-algebras. 
Hence, Prop.\ 4.7 of \cite{Getzler} implies that their images under $\sMC_{\bul}(-)$ are fibrations of simplicial sets. The inverse limit $\varprojlim \maps \tow(\sSet) \to \sSet$ of this morphism of towers is $\sMC_{\bul}(\Phi) \maps \sMC_{\bul}(L) \to \sMC_{\bul}(\ti{L})$. The functor $\varprojlim$ is right Quillen \cite[Ch.\ VI, Def.\ 1.7]{GJ}. 
Hence, to show $\sMC_{\bul}(\Phi)$ is a fibration, it is sufficient to show
that the morphism of towers \eqref{eq:S_tow} is a fibration. By
definition, this means we must show, for each $n >1$, that the morphism
induced by the universal property in the pullback diagram:
\[
\begin{tikzpicture}[descr/.style={fill=white,inner sep=2.5pt},baseline=(current  bounding  box.center)]
\matrix (m) [matrix of math nodes, row sep=2em,column sep=3em,
  ampersand replacement=\&]
  {  
\sMC_\bul(L/\cF_n L) \& \sMC_{\bul}(\ti{L}/\cF_n \ti{L}) \times_{\sMC_{\bul}(\ti{L}/\cF_{n-1} \ti{L})} \sMC_{\bul}(L/\cF_{n-1} L) 
\& \sMC_{\bul}(L/\cF_{n-1} L)\\
\& \sMC_{\bul}(\ti{L}/\cF_n \ti{L}) \& \sMC_{\bul}(\ti{L}/\cF_{n-1} \ti{L})\\
}; 
  \path[->,font=\scriptsize] 
   (m-1-1) edge node[auto] {$ (\ov{\Phi}^{(\bul)}_{\ast},p^{(\bul)}_{n \ast} )$} (m-1-2)
   (m-1-2) edge node[auto] {$$} (m-1-3)
   (m-1-2) edge node[auto,swap] {$$} (m-2-2)
   (m-1-3) edge node[auto] {$\ov{\Phi}^{(\bul)}_{\ast}$} (m-2-3)
   (m-2-2) edge node[auto,swap] {$\ti{p}^{(\bul)}_{n\ast}$} (m-2-3)
   (m-1-1) edge [bend left=25] node[auto]{$p^{(\bul)}_{n\ast}$} (m-1-3)
   (m-1-1) edge [bend right=15,swap] node[auto]{$\ov{\Phi}^{(\bul)}_{\ast}$} (m-2-2)

  ;
  \begin{scope}[shift=($(m-1-2)!.4!(m-2-3)$)]
  \draw +(-0.25,0) -- +(0,0)  -- +(0,0.25);
  \end{scope}
\end{tikzpicture}
\]
is a fibration of simplicial sets \cite[Ch.\ VI, Def.\ 1.1]{GJ}.
So suppose we are given a horn  $\gamma \maps   \horn{m}{k} \to \sMC_{\bul}(L/\cF_n L)$ and commuting diagrams
\[
\begin{tikzpicture}[descr/.style={fill=white,inner sep=2.5pt},baseline=(current  bounding  box.center)]
\matrix (m) [matrix of math nodes, row sep=2em,column sep=3em,
  ampersand replacement=\&]
  {  
\horn{m}{k} \& \sMC_{\bul}(L/\cF_n L) \\
\Delta^m \& \sMC_{\bul}(\ti{L}/\cF_n \ti{L})\\
}; 
  \path[->,font=\scriptsize] 
   (m-1-1) edge node[auto] {$\gamma$} (m-1-2)
   (m-1-1) edge node[auto] {$$} (m-2-1)
   (m-2-1) edge node[auto] {$\ti{\beta}$} (m-2-2)
   (m-1-2) edge node[auto] {$\ov{\Phi}^{(\bul)}_{\ast}$} (m-2-2)
  ;
\end{tikzpicture}
\qquad
\begin{tikzpicture}[descr/.style={fill=white,inner sep=2.5pt},baseline=(current  bounding  box.center)]
\matrix (m) [matrix of math nodes, row sep=2em,column sep=3em,
  ampersand replacement=\&]
  {  
\horn{m}{k} \& \sMC_{\bul}(L/\cF_n L) \\
\Delta^m \& \sMC_{\bul}(L/\cF_{n-1} L)\\
}; 
  \path[->,font=\scriptsize] 
   (m-1-1) edge node[auto] {$\gamma$} (m-1-2)
   (m-1-1) edge node[auto] {$$} (m-2-1)
   (m-2-1) edge node[auto] {${\beta}$} (m-2-2)

  ;
  \path[->>,font=\scriptsize] 
   (m-1-2) edge node[auto] {${p}^{(\bul)}_{n\ast}$} (m-2-2)
;
\end{tikzpicture}
\qquad
\begin{tikzpicture}[descr/.style={fill=white,inner sep=2.5pt},baseline=(current  bounding  box.center)]
\matrix (m) [matrix of math nodes, row sep=2em,column sep=3em,
  ampersand replacement=\&]
  {  
\Delta^m \& \sMC_{\bul}(L/\cF_{n-1} L)\\
\sMC_{\bul}(\ti{L}/\cF_n \ti{L}) \& \sMC_{\bul}(\ti{L}/\cF_{n-1} \ti{L})\\
}; 
  \path[->,font=\scriptsize] 
   (m-1-1) edge node[auto] {$\beta$} (m-1-2)
   (m-1-1) edge node[auto,swap] {$\ti{\beta}$} (m-2-1)
   (m-1-2) edge node[auto] {$\ov{\Phi}^{(\bul)}_{\ast}$} (m-2-2)
  ;

  \path[->>,font=\scriptsize] 
   (m-2-1) edge node[auto] {$\ti{p}^{(\bul)}_{n\ast}$} (m-2-2)
;
\end{tikzpicture}
\]
We need to produce an $m$-simplex $\alpha \maps \Delta^m \to \sMC_{\bul}(L/\cF_nL)$ which fills the horn $\gamma$ and satisfies $\ov{\Phi}^{(\bul)}_{\ast} \alpha = \ti{\beta}$ and
${p}^{(\bul)}_{n\ast} \alpha = \beta$. Since $p^{(\bul)}_{n\ast}$ is a fibration, there exists an $m$-simplex $\theta$  lifting $\beta$:
\[
\begin{tikzpicture}[descr/.style={fill=white,inner sep=2.5pt},baseline=(current  bounding  box.center)]
\matrix (m) [matrix of math nodes, row sep=2em,column sep=3em,
  ampersand replacement=\&]
  {  
\horn{m}{k} \& \sMC_{\bul}(L/\cF_n L) \\
\Delta^m \& \sMC_{\bul}(L/\cF_{n-1} L)\\
}; 
  \path[->,font=\scriptsize] 
   (m-1-1) edge node[auto] {$\gamma$} (m-1-2)
   (m-1-1) edge node[auto] {$$} (m-2-1)
   (m-2-1) edge node[auto,swap] {${\beta}$} (m-2-2)
   (m-2-1) edge node[auto] {${\theta}$} (m-1-2)

  ;
  \path[->>,font=\scriptsize] 
   (m-1-2) edge node[auto] {${p}^{(\bul)}_{n\ast}$} (m-2-2)
;
\end{tikzpicture}
\]
but there is no guarantee that $\ov{\Phi}^{(\bul)}_{\ast}(\theta)=\ti{\beta}$. However, note that  
the $m$-simplex 
\begin{equation} \label{eq:eta}
\eta:= \ov{\Phi}^{(\bul)}_{\ast}(\theta)- \ti{\beta}
\end{equation}
of the simplicial vector space $\ti{L}/\cF_n
\ti{L}\tensor\Omega_{\bul}$ lies in the kernel of the linear map
$\ti{p}^{(m)}_{n}$. We now observe that the fibration $\Phi \maps L
\to \ti{L}$ induces a map between the short exact sequences
\eqref{eq:ses} of nilpotent $L_\infty$-algebras:

\begin{equation}
\begin{tikzpicture}[descr/.style={fill=white,inner sep=2.5pt},baseline=(current  bounding  box.center)]
\matrix (m) [matrix of math nodes, row sep=2em,column sep=3em,
  ampersand replacement=\&]
  {  
\cF_{n-1}L/\cF_n L \&  L / \cF_{n}L  \&  L / \cF_{n-1}L \\
\cF_{n-1}\ti{L}/\cF_n \ti{L} \&  \ti{L} / \cF_{n} \ti{L}  \&  \ti{L} / \cF_{n-1} \ti{L} \\
}; 

\path[->,font=\scriptsize] 
  (m-1-1) edge node[auto] {$$} (m-1-2)
  (m-2-1) edge node[auto] {$$} (m-2-2)
  (m-1-2) edge node[auto,swap] {$ \ov{\Phi} $} (m-2-2)
  (m-1-3) edge node[auto] {$ \ov{\Phi} $} (m-2-3)
  (m-1-1) edge node[auto,swap] {$ \ov{\cF_{n-1} \Phi} $} (m-2-1)
;
\path[->>,font=\scriptsize] 
(m-1-2) edge node[auto] {$p_n$} (m-1-3)
(m-2-2) edge node[auto] {$\ti{p}_n$} (m-2-3)
;
\end{tikzpicture}
\end{equation}
It follows from the compatibility of $\Phi$ with the filtrations, that
$\ov{\cF_{n-1} \Phi}$ above is simply the linear term of the morphism
$\ov{\Phi}$ restricted to the subspace $\cF_{n-1}L/\cF_nL$. Moreover, since
$\Phi$ is a fibration, $\ov{\cF_{n-1} \Phi}$ is surjective. Hence,
$\ov{\cF_{n-1} \Phi}$ is a strict fibration between abelian
$L_\infty$-algebras, and so Prop.\ 4.7 of \cite{Getzler} implies that
the corresponding map in the diagram of simplicial sets below is a
fibration:
\begin{equation}
\begin{tikzpicture}[descr/.style={fill=white,inner sep=2.5pt},baseline=(current  bounding  box.center)]
\matrix (m) [matrix of math nodes, row sep=2em,column sep=3em,
  ampersand replacement=\&]
  {  
\sMC_{\bul} \Bigl( \cF_{n-1}L/\cF_n L \Bigr) \&  \sMC_{\bul} \Bigl( L / \cF_{n}L \Bigr)  \&  \sMC_{\bul} \Bigl( L / \cF_{n-1}L \Bigr) \\
\sMC_{\bul} \Bigl( \cF_{n-1}\ti{L}/\cF_n \ti{L} \Bigr) \&  \sMC_{\bul} \Bigl( \ti{L} / \cF_{n} \ti{L} \Bigr)  \&  \sMC_{\bul} \Bigl( \ti{L} / \cF_{n-1} \ti{L} \Bigr) \\
}; 

\path[->,font=\scriptsize] 
  (m-1-1) edge node[auto] {$$} (m-1-2)
  (m-2-1) edge node[auto] {$$} (m-2-2)
  (m-1-2) edge node[auto,swap] {$ \ov{\Phi}^{(\bul)}_{\ast} $} (m-2-2)
  (m-1-3) edge node[auto] {$ \ov{\Phi}^{(\bul)}_{\ast} $} (m-2-3)
;
\path[->>,font=\scriptsize] 
(m-1-2) edge node[auto] {$p^{(\bul)}_{n \ast}$} (m-1-3)
(m-2-2) edge node[auto] {$\ti{p}^{(\bul)}_{n \ast}$} (m-2-3)
(m-1-1) edge node[auto,swap] {$ \ov{\cF_{n-1} \Phi}^{(\bul)}_{\ast} $} (m-2-1)
;
\end{tikzpicture}
\end{equation}
A straightforward calculation shows that the vector $\eta$ \eqref{eq:eta} is in fact a $m$-simplex of
$\sMC_{\bul} \Bigl( \cF_{n-1}\ti{L}/\cF_n \ti{L} \Bigr)$, whose restriction to the horn $\horn{m}{k}$ vanishes.
Hence, there exists a lift $\lambda \maps \Delta^m \to \sMC_{\bul} \Bigl( \cF_{n-1}L/\cF_n L \Bigr)$ of $\eta$ through $\ov{\cF_{n-1} \Phi}^{(\bul)}_{\ast}$. 

One can then show via a series of technical lemmas that
$\alpha=\lambda + \theta$ is a $m$-simplex of $\sMC_{\bul}(L/\cF_nL)$ which
fills the horn $\gamma$ and satisfies both $\ov{\Phi}^{(\bul)}_{\ast}
\alpha = \ti{\beta}$ and ${p}^{(\bul)}_{n\ast} \alpha = \beta$.
Hence, the morphism of towers \eqref{eq:S_tow} is a fibration in $\tow(\sSet)$,
and we conclude that  $\sMC_{\bul}(\Phi) \maps \sMC_{\bul}(L) \to \sMC_{\bul}(\ti{L})$ is a fibration
of simplicial sets.




\section{Homotopy transfer theorem} \label{sec:HTT}
For this section, we follow the conventions presented in Sections 1 and 2 of
\cite{HAforms}. We refer the reader there for further background on dg
operads and homotopy algebras.
Let $\cC$ be a dg cooperad with a co-augmentation $\bar{\cC}$ that is equipped with 
a compatible cocomplete ascending filtration:
\begin{equation}
\label{eq:cooperad_filtr}
0 =  \cF^0 \bar{\cC} \subset \cF^1 \bar{\cC} \subset  \cF^2 \bar{\cC} \subset  \cF^3 \bar{\cC} \subset \dots 
\end{equation}
Any co-augmented cooperad satisfying $\cC(0)=0$, $\cC(1)=\Bbbk$, for example, admits 
such a filtration (by arity). $\Cobar(\cC)$ algebra structures on a cochain
complex $(A,\del_A)$ are in one--to--one correspondence with
codifferentials $Q$ on the cofree coalgebra $\cC(A)=\bigoplus_{n \geq
  0}\Bigl(\cC(n)\tensor A^{\tensor n} \Bigr)_{S_n}$ which satisfy
$Q\vert_{A}=\del$. Homotopy algebras such as $L_\infty$, $A_\infty$, and $C_\infty$ algebras
are all examples of $\Cobar(\cC)$ algebras of this kind. 
A morphism (or more precisely ``$\infty$-morphism'')
$F \maps (A,Q_A) \to (B,Q_B)$ between  $\Cobar(\cC)$ algebras is morphism
between the corresponding dg coalgebras $F \maps 
\bigl( \cC(A), \del_A + Q_A \bigr) \to  \bigl( \cC(B), \del_B + Q_B
\bigr)$. Such a morphism is an {\bf $\infty$-quasi-isomorphism} iff
its linear term $\pr_B F \vert_A \maps (A,\del_A) \to (B,\del_B)$ is a
quasi--isomorphism of chain complexes. 

Given a cochain complex $(A,\del_A)$, one can construct a dg Lie
algebra $\Conv(\bar{\cC},\End_A)$ whose Maurer--Cartan elements are in
one--to--one correspondence with $\Cobar(\cC)$ structures on $(A,\del_A)$.
The underlying complex of $\Conv(\bar{\cC},\End_A)$ can be identified
with the complex of linear maps $\Hom(\bar{\cC}(A),A)$.
The filtration \eqref{eq:cooperad_filtr} induces a complete descending
filtration on $\Conv(\bar{\cC},\End_A)$ which is compatible with the
dg Lie structure. Hence, the desuspension $\bs^{-1}
\Conv(\bar{\cC},\End_A)$ is a filtered $L_\infty$-algebra in our sense.

We now indulge in some minor pedantry by presenting the
well-known Homotopy Transfer Theorem in the following way.
Let $(B,Q_B)$ be a $\Cobar(\cC)$-algebra, $(A,\del)$ a cochain
complex, and $\phi \maps A \to B$ a quasi--isomorphism of cochain
complexes. One asks whether the structure on $B$ can be transferred
through $\phi$ to a homotopy equivalent structure on $A$. A solution
to the \textbf{homotopy transfer problem} is a 
$\Cobar(\cC)$-structure $Q_A$ on $A$, and
a $\infty$-quasi-isomorphism $F \maps (A,Q_A) \weq (B,Q_B)$ of
$\Cobar(\cC)$-algebras such that $\pr_BF \vert_A =\phi$.


Solutions to the homotopy transfer problem correspond to certain MC
elements of a filtered $L_\infty$--algebra. The cochain complex
\begin{equation} \label{eq:cyl} \Cyl(\cC,A,B):=
  \bs^{-1}\Hom(\bar{\cC}(A),A) \oplus \Hom(\cC(A),B) \oplus
\bs^{-1}  \Hom(\bar{\cC}(B),B)
\end{equation}
can be equipped with a (shifted) $L_\infty$-structure induced by:
(1) the convolution Lie brackets on $\Hom(\bar{\cC}(A),A)$ and
$\Hom(\bar{\cC}(B),B)$, and (2) pre and post composition of elements
of $\Hom(\cC(A),B)$ with elements of $\Hom(\bar{\cC}(A),A)$ and
$\Hom(\bar{\cC}(B),B)$, respectively. (See Sec.\ 3.1 in \cite{DW} for
the details.)


As shown in Sec.\ 3.2 of \cite{DW},
the $L_\infty$-structure on $\Cyl(\cC,A,B)$ is such that its
MC elements are triples $(Q_A,F,Q_B)$, where $Q_A$ and $Q_B$ are
$\Cobar(\cC)$ structures on $A$ and $B$, respectively, and $F$ is a
$\infty$-morphism between them. In particular, if $\phi \maps A \to B$ is a
chain map, then $\alpha_\phi=(0,\phi,0)$ is a MC element in
$\Cyl(\cC,A,B)$, where ``0'' denotes the trivial $\Cobar(\cC)$ structure.

We can therefore twist, as described in Sec.\ \ref{sec:MC},  by the MC element $\alpha_\phi$ 
to obtain a new
$L_\infty$-algebra $\Cyl(\cC,A,B)^{\alpha_\phi}$. The graded subspace
\begin{equation}\label{eq:sub_cyl}
  \ov{\Cyl}(\cC, A, B)^{\alpha_\phi} : = \bs^{-1} \Hom(\bar{\cC}(A), A) ~\oplus~ \Hom(\bar{\cC}(A), B)~ \oplus~ \bs^{-1} \Hom(\bar{\cC}(B), B)
\end{equation}
is equipped with a filtration induced by the filtration on
$\cC$. Restricting the $L_\infty$ structure on ${\Cyl}(\cC, A, B)^{\alpha_\phi}$ to
$\ov{\Cyl}(\cC, A, B)^{\alpha_\phi}$ makes the latter into a  filtered $L_\infty$-algebra.
The MC elements of $\ov{\Cyl}(\cC, A, B)^{\alpha_\phi}$ are those MC
elements $(Q_A,F,Q_B)$ of $\Cyl(\cC,A,B)$ such that $\pr_B F \vert_A = \phi$.

We have the following proposition. (See Prop.\ 3.2 in \cite{DW}).

\begin{proposition} \label{prop:proj}
The canonical projection of cochain complexes
\begin{equation} \label{eq:proj1}
\pi_B \maps \bs^{-1} \Hom(\bar{\cC}(A), A) ~\oplus~ \Hom(\bar{\cC}(A),
B)~ \oplus~ \bs^{-1} \Hom(\bar{\cC}(B), B) \to \bs^{-1} \Hom(\bar{\cC}(B), B)
\end{equation}
lifts to a (strict) acyclic fibration of filtered $L_\infty$-algebras:
\[
\pi_B \maps \ov{\Cyl}(\cC, A, B)^{\alpha_\phi} \trivfib \bs^{-1} \Conv(\bar{\cC},\End_B)
\]
\end{proposition}

We can now express the homotopy transfer theorem as a simple corollary:
\begin{corollary}[Homotopy Transfer Theorem]
  Let $(B,Q_B)$ be a $\Cobar(\cC)$-algebra, $(A,\del)$ a cochain
  complex, and $\phi \maps A \to B$ a quasi--isomorphism of cochain
  complexes. The solutions to the corresponding homotopy
  transfer problem are in one--to--one correspondence with 
the objects of a  sub $\infty$--groupoid
\[  
\fF_{Q_B} \subseteq \sMC_{\bul}\Bigl ( \ov{\Cyl}(\cC, A, B)^{\alpha_\phi} \Bigr).
\]
Furthermore,
\begin{enumerate}
\item (existence) $\fF_{Q_B}$ is non--empty, and

\item (uniqueness) $\fF_{Q_B}$ is contractible.
\end{enumerate}
\end{corollary}

\begin{proof}
All statements follow from Theorems \ref{thm:GM} and \ref{thm:fib}, which imply that 
\begin{equation}\label{eq:cor1}
\sMC_{\bul}(\pi_B) \maps \sMC_{\bul}\Bigl ( \ov{\Cyl}(\cC, A, B)^{\alpha_\phi} \Bigr)
\trivfib 
\sMC_{\bul} \Bigl (\Conv(\bar{\cC},\End_B) \Bigr)
\end{equation}
is an acyclic fibration of Kan complexes. Indeed, 
we define $\fF_{Q_B}$ as the fiber of $\sMC_{\bul}(\pi_B)$ over the object
$Q_B \in \sMC_{0} \Bigl (\Conv(\bar{\cC},\End_B) \Bigr)$.
Since $\sMC_{\bul}(\pi_B)$ is a Kan fibration, $\fF_{Q_B}$ is a
$\infty$-groupoid. Objects of
$\fF_{Q_B}$ are those MC elements of $\ov{\Cyl}(\cC, A, B)^{\alpha_\phi}$
which are of the form $(Q_A,F,Q_B)$, and hence are solutions to the
homotopy transfer problem.

Since $\sMC_{\bul}(\pi_B)$ is an acyclic fibration, it satisfies the right
lifting property with respect to the inclusion $\emptyset=\del
\Delta^0 \subseteq \Delta^0$. Hence, $\sMC_{\bul}(\pi_B)$ is surjective on
objects. This proves statement (1). Statement (2) follows
from the long exact sequence of homotopy groups.

\end{proof}

Let us conclude by mentioning the difference between the above formulation of the
Homotopy Transfer Theorem and the one given in Section 5 of our
previous work \cite{HAforms} with V.\ Dolgushev. There we only had
Thm. \ref{thm:GM} to use, and not Thm.\ \ref{thm:fib}. Hence,
we proved a slight variant of the transfer theorem \cite [Thm.\ 5.1]{HAforms}.
We defined a solution to the homotopy transfer problem as a triple
$(Q_A,F,\ti{Q}_B)$, where $Q_A$ is a $\Cobar(\cC)$ algebra structure
on $A$, $\ti{Q}_B$ is a $\Cobar(\cC)$ algebra structure on $B$ homotopy
equivalent to the original structure $Q_B$, and $F \maps (A,Q_A) \to
(B,\ti{Q}_B)$ is a $\infty$-quasi-isomorphism whose linear term is
$\phi$. We used the fact that $\sMC_{\bul}(\pi_B)$ is a weak equivalence, and
therefore gives a bijection
 \[
 \pi_0 \left (\sMC_{\bul}\Bigl ( \ov{\Cyl}(\cC, A, B)^{\alpha_\phi} \Bigr)
 \right) \cong \pi_0 \left( \sMC_{\bul} \Bigl (\Conv(\bar{\cC},\End_B) \Bigr) \right),
 \]
to conclude that such a solution $(Q_A,F,\ti{Q}_B)$ exists. It is easy
to see that objects of the \textit{homotopy fiber} of $\sMC_{\bul}(\pi_B)$ over the
vertex $Q_B$ are pairs consisting of a solution $(Q_A,F,\ti{Q}_B)$
to this variant of the transfer problem, and an equivalence from $\ti{Q}_B$ to $Q_B$.

\begin{acknowledgement}
I would like to acknowledge support by an AMS--Simons Travel Grant, and I thank Vasily Dolgushev and Bruno Vallette for helpful discussions regarding this work.
I would also like to thank the organizers of the MATRIX Institute program ``Higher Structures in Geometry and Physics'' for an excellent workshop and conference.

\end{acknowledgement}


\begin{thebibliography}{10}
\providecommand{\url}[1]{{#1}}
\providecommand{\urlprefix}{URL }
\expandafter\ifx\csname urlstyle\endcsname\relax
  \providecommand{\doi}[1]{DOI~\discretionary{}{}{}#1}\else
  \providecommand{\doi}{DOI~\discretionary{}{}{}\begingroup
  \urlstyle{rm}\Url}\fi

\bibitem{BM}
Berger, C., Moerdijk, I.: Axiomatic homotopy theory for operads.
\newblock Comment. Math. Helv. \textbf{78}(4), 805--831 (2003)

\bibitem{Berglund2}
Berglund, A.: Homological perturbation theory for algebras over operads.
\newblock Algebr. Geom. Topol. \textbf{14}(5), 2511--2548 (2014)

\bibitem{Berglund}
Berglund, A.: Rational homotopy theory of mapping spaces via {L}ie theory for
  {$L_\infty$}-algebras.
\newblock Homology Homotopy Appl. \textbf{17}(2), 343--369 (2015)

\bibitem{Buijs}
Buijs, U., F{\'e}lix, Y., Murillo, A.: {$L_\infty$} models of based mapping
  spaces.
\newblock J. Math. Soc. Japan \textbf{63}(2), 503--524 (2011)

\bibitem{DW}
Dolgushev, V., Willwacher, T.: The deformation complex is a homotopy invariant
  of a homotopy algebra.
\newblock In: Developments and retrospectives in {L}ie theory, \emph{Dev.
  Math.}, vol.~38, pp. 137--158. Springer, Cham (2014)

\bibitem{HAforms}
Dolgushev, V.A., Hoffnung, A.E., Rogers, C.L.: What do homotopy algebras form?
\newblock Adv. Math. \textbf{274}, 562--605 (2015)

\bibitem{GM_Theorem}
Dolgushev, V.A., Rogers, C.L.: A version of the {G}oldman-{M}illson theorem for
  filtered {$L_\infty$}-algebras.
\newblock J. Algebra \textbf{430}, 260--302 (2015)

\bibitem{Enhanced}
Dolgushev, V.A., Rogers, C.L.: On an enhancement of the category of shifted
  ${L}_\infty$-algebras.
\newblock Applied Categorical Structures pp. 1--15 (2016)

\bibitem{Dotsenko}
Dotsenko, V., Poncin, N.: A tale of three homotopies.
\newblock Applied Categorical Structures pp. 1--29 (2015)

\bibitem{Getzler}
Getzler, E.: Lie theory for nilpotent {$L_\infty$}-algebras.
\newblock Ann. of Math. (2) \textbf{170}(1), 271--301 (2009)

\bibitem{GJ}
Goerss, P.G., Jardine, J.F.: Simplicial homotopy theory, \emph{Progress in
  Mathematics}, vol. 174.
\newblock Birkh\"auser Verlag, Basel (1999)

\bibitem{Hinich}
Hinich, V.: Descent of {D}eligne groupoids.
\newblock Internat. Math. Res. Notices. (5), 223--239 (1997)

\bibitem{Lazarev}
Lazarev, A.: Maurer-{C}artan moduli and models for function spaces.
\newblock Adv. Math. \textbf{235}, 296--320 (2013)

\bibitem{LV}
Loday, J.L., Vallette, B.: Algebraic operads, \emph{Grundlehren der
  Mathematischen Wissenschaften [Fundamental Principles of Mathematical
  Sciences]}, vol. 346.
\newblock Springer, Heidelberg (2012)

\bibitem{Markl}
Markl, M.: Homotopy algebras are homotopy algebras.
\newblock Forum Math. \textbf{16}(1), 129--160 (2004)

\bibitem{Yalin}
Yalin, S.: Maurer--{C}artan spaces of filtered ${L}_\infty$-algebras.
\newblock J. Homotopy Relat. Struct \textbf{11}, 375--407 (2016)

\end{thebibliography}

\end{document}